\documentclass[12pt,reqno]{amsart}

\usepackage{amstext}
\usepackage{amsthm}
\usepackage{amsmath}
\usepackage{amssymb}
\usepackage{latexsym}
\usepackage{amsfonts}
\usepackage{graphicx}
\usepackage{texdraw}
\usepackage{graphpap}
\allowdisplaybreaks

\usepackage[pagebackref,hypertexnames=false, colorlinks, citecolor=red, linkcolor=red]{hyperref} %

\usepackage[backrefs]{amsrefs}

\input txdtools

\begin{document}

\newcommand{\ci}[1]{_{ {}_{\scriptstyle #1}}}

\newcommand{\norm}[1]{\ensuremath{\left\|#1\right\|}}
\newcommand{\abs}[1]{\ensuremath{\left\vert#1\right\vert}}
\newcommand{\ip}[2]{\ensuremath{\left\langle#1,#2\right\rangle}}
\newcommand{\p}{\ensuremath{\partial}}
\newcommand{\pr}{\mathcal{P}}

\newcommand{\pbar}{\ensuremath{\bar{\partial}}}
\newcommand{\db}{\overline\partial}
\newcommand{\D}{\mathbb{D}}
\newcommand{\B}{\mathbb{B}}
\newcommand{\Sp}{\mathbb{S}}
\newcommand{\T}{\mathbb{T}}
\newcommand{\R}{\mathbb{R}}
\newcommand{\Z}{\mathbb{Z}}
\newcommand{\C}{\mathbb{C}}
\newcommand{\N}{\mathbb{N}}
\newcommand{\scrH}{\mathcal{H}}
\newcommand{\scrL}{\mathcal{L}}
\newcommand{\td}{\widetilde}

\newcommand{\La}{\langle }
\newcommand{\Ra}{\rangle }
\newcommand{\rk}{\operatorname{rk}}
\newcommand{\card}{\operatorname{card}}
\newcommand{\ran}{\operatorname{Ran}}
\newcommand{\osc}{\operatorname{OSC}}
\newcommand{\im}{\operatorname{Im}}
\newcommand{\re}{\operatorname{Re}}
\newcommand{\tr}{\operatorname{tr}}
\newcommand{\vf}{\varphi}
\newcommand{\f}[2]{\ensuremath{\frac{#1}{#2}}}

\newcommand{\kzp}{k_z}
\newcommand{\klp}{k_{\lambda_i}^{(p,\alpha)}}
\newcommand{\kz}{k_z}
\newcommand{\kl}{k_{\lambda_i}^{(\alpha)}}
\newcommand{\TTp}{\mathcal{T}_p}
\newcommand{\Cd}{\mathbb{C}^{d}}
\newcommand{\cd}{\mathbb{C}^{d}}
\newcommand{\Bn}{\mathbb{B}^{n}}
\newcommand{\Cn}{\mathbb{C}^{n}}
\newcommand{\Langle}{\left\langle}
\newcommand{\Rangle}{\right\rangle}
\newcommand{\lp}{\left(}
\newcommand{\rp}{\right)}
\newcommand{\dva}{dv_{\alpha}}
\newcommand{\ata}{A_{\alpha}^{2}}
\newcommand{\lpa}{L_{\alpha}^{p}}
\providecommand{\apa}{A_{\alpha}^{p}}
\providecommand{\K}[2]{K_{#1}}
\providecommand{\nk}[2]{k_{#1}}
\providecommand{\diam}{\textnormal{diam}}


\newcommand{\entrylabel}[1]{\mbox{#1}\hfill}

\newenvironment{entry}
{\begin{list}{X}%
  {\renewcommand{\makelabel}{\entrylabel}%
      \setlength{\labelwidth}{55pt}%
      \setlength{\leftmargin}{\labelwidth}
      \addtolength{\leftmargin}{\labelsep}%
   }%
}%
{\end{list}}


\numberwithin{equation}{section}

\newtheorem{thm}{Theorem}[section]
\newtheorem{lm}[thm]{Lemma}
\newtheorem{lem}[thm]{Lemma}
\newtheorem{cor}[thm]{Corollary}
\newtheorem{conj}[thm]{Conjecture}
\newtheorem{prob}[thm]{Problem}
\newtheorem{prop}[thm]{Proposition}
\newtheorem*{prop*}{Proposition}
\newtheorem{defn}[thm]{Definition}

\theoremstyle{remark}
\newtheorem{rem}[thm]{Remark}
\newtheorem*{rem*}{Remark}

\hyphenation{geo-me-tric}

\title[Compact Operators on Vector--Valued Bergman Space]
{compact operators on vector--valued bergman space via the 
berezin transform}

\author[R. Rahm]{Robert S. Rahm$^\dagger$}
\address{Robert S. Rahm, School of Mathematics\\ Georgia Institute of Technology
\\ 686 Cherry Street\\ Atlanta, GA USA 30332-0160}
\email{rrahm3@math.gatech.edu}


\begin{abstract}
In this paper, we characterise compactness of finite sums of finite products of 
Toeplitz operators acting on the
\begin{math}\cd\end{math}-valued weighted Bergman Space, denoted 
\begin{math}A_{\alpha}^{p}(\Bn,\cd)\end{math}. The main result shows that a 
finite sum of finite product of Toeplitz
operators acting on \begin{math}A_{\alpha}^{p}(\Bn,\cd)\end{math} is compact if 
and only if its Berezin transform
vanishes on the boundary of the ball. 
\end{abstract}

\maketitle
\section{Introduction and Statement of Main Results}
\subsection{Definition of the Spaces $L_{\alpha}^{p}(\Bn;\Cd)$ and 
$A_{\alpha}^{p}(\Bn;\Cd)$}

Let \begin{math}\Bn\end{math} denote the open unit ball in 
\begin{math}\C^{n}\end{math}. For \begin{math}d\in\N
\end{math}, a function, \begin{math}f\end{math}, defined on 
\begin{math}\Bn\end{math} and taking values in \begin{math}\Cd\end{math}
(that is, \begin{math}f\end{math} is vector-valued)
is said to be measurable if 
\begin{math}z\mapsto \Langle f(z),e\Rangle_{\Cd}
\end{math} is measurable for every \begin{math}e\in\Cd \end{math}. Since
$\Cd$ is finite dimensional, this is the same as requiring that $f$ 
be measurable in each coordinate. For a vector, $v\in\Cd$, let 
$\norm{v}_{2}$ denote the usual euclidean norm on $\Cd$. That is
$\norm{v}_{2}^{2}:=\ip{v}{v}_{\Cd}$, where, of course, $\ip{}{}_{\Cd}$ is the
standard inner product on $\Cd$.
For \begin{math} \alpha > -1
\end{math}, let 
\begin{align*}
\dva (z) := c_{\alpha}(1-|z|^{2})^{\alpha}dV (z) 
\end{align*}
where $dV$ is volume measure on $\Bn$ and \begin{math}c_{\alpha} \end{math} is a
constant such that \begin{math}\int_{\Bn}\dva (z) = 1 \end{math}. Define 
\begin{math} L_{\alpha}^{2}(\Bn;\Cd)\end{math} to be the set of 
all measurable functions on \begin{math}\Bn\end{math} taking values in 
\begin{math}\Cd\end{math} such that 
\begin{align*}
\norm{f}_{L_{\alpha}^{2}(\Bn;\Cd)}^{2}
:=\int_{\Bn}\ip{f(z)}{f(z)}_{\Cd} \dva (z) 
=\int_{\Bn}\norm{f(z)}_{2}^{2}dv_{\alpha}(z) < \infty.
\end{align*}
It should be noted that \begin{math}L_{\alpha}^{2}(\Bn;\Cd)\end{math} is a 
Hilbert Space with inner product:
\begin{align*}
\ip{f}{g}_{L_{\alpha}^{2}(\Bn;\Cd)} := \int_{\Bn}\ip{f(z)}{g(z)}_{\Cd}\dva (z).
\end{align*}

\noindent It is obvious that this is an inner product and that 
\begin{math}\norm{f}_{L_{\alpha}^{2}(\Bn;\Cd)}^{2}
= \ip{f}{f}_{L_{\alpha}^{2}(\Bn;\Cd)}
\end{math}. 

Similarly, a function \begin{math}f:\Bn\to\Cd\end{math} is said to be 
holomorphic if \begin{math}z\mapsto\Langle f(z),e\Rangle_{\Cd} \end{math} is a 
holomorphic function for every \begin{math}e\in\C^{d}\end{math}. 
(Similarly, this 
is the same as requiring that $f$ be holomorphic in each coordinate.)
Define \begin{math}A_{\alpha}^{2}(\Bn;\Cd)\end{math} to be the 
set of holomorhic functions on $\Bn$ that are also in 
\begin{math}L_{\alpha}^{2}(\Bn;\Cd)\end{math}. Additionally, define 
\begin{math}\mathcal{L}(A_{\alpha}^{2}(\Bn;\Cd))\end{math} be the set of 
bounded linear operators from $A_{\alpha}^{2}(\Bn;\Cd)$ to 
$A_{\alpha}^{2}(\Bn;\Cd)$.

\subsection{Background for the Scalar--Valued Case}
For the moment, let \begin{math}d=1\end{math}.
Recall the reproducing kernel:
\begin{align*}
K_{z}^{(\alpha)}(w)=K^{(\alpha)}(z,w)
:=\frac{1}{(1-\overline{z}w)^{n+1+\alpha}}.
\end{align*} 

\noindent That is, if \begin{math}f\in A_{\alpha}^{2}(\Bn;\C)\end{math} 
there holds: 
\begin{align*}
f(z) = \ip{f}{\K{z}{2}}_{A_{\alpha}^{2}(\Bn;\C)}=
\int_{\Bn}\frac{f(w)}{(1-z\overline{w})^{n+1+\alpha}}\dva (w).
\end{align*}

\noindent Recall also the normalized reproducing kernels, $k_{z}^{(2,\alpha)}$,
normalized so that 
\begin{math}\norm{k_{z}^{(2,\alpha)}}_{L_{\alpha}^{2}(\Bn;\C)}=1\end{math}.
A simple calculation shows:
\begin{align*}
k_{z}^{(2,\alpha)}(w)=k^{(2,\alpha)}(z,w)
:=\norm{K_z^{(\alpha)}}^{-1}_{L_{\alpha}^{2}(\Bn;\C)}K_{z}^{(\alpha)}(w)=
\frac{(1-|z|^{2})^{\frac{n+1+\alpha}{2}}}
{(1-\overline{z}w)^{n+1+\alpha}}.
\end{align*} 

The reproducing kernels allow us to explicitly write the orthogonal 
projection, $P_{\alpha}$, from \begin{math}L_{2}^{\alpha}(\Bn;\C) \end{math} to 
\begin{math}A_{\alpha}^{2}(\Bn;\C) \end{math}:
\begin{align*}
(P_{\alpha}f)(z)=\ip{f}{K_{z}^{(\alpha)}}_{L_{\alpha}^{2}(\Bn;\C)}.
\end{align*}
Let \begin{math}\varphi\in L^{\infty}(\Bn)\end{math}. The Toeplitz operator with 
symbol \begin{math}\varphi\end{math} is defined to be:
\begin{align*}
T_{\varphi}:=P_{\alpha}M_{\varphi},
\end{align*}
where \begin{math}M_{\varphi}\end{math} is the multiplication operator. So there
holds: \begin{math}(T_{\varphi}f)(z) = 
\ip{\varphi f}{K_{z}^{(\alpha)}}_{L_{\alpha}^{2}(\Bn;\C)}\end{math}. 
Recall that the 
Berezin Transform of \begin{math}T\end{math}, denoted 
\begin{math}\widetilde{T} \end{math}, is a function on \begin{math}\Bn\end{math}
defined by the formula: \begin{math} \widetilde{T}(\lambda) = 
\ip{T k_{\lambda}^{(2,\alpha)}}
{k_{\lambda}^{(2,\alpha)}}_{A_{\alpha}^{2}(\Bn;\C)} \end{math}.

\subsection{Generalization to Vector--Valued Case}
Now, we consider \begin{math}d\in\N\end{math} and $d>1$.
The preceding discussion can be carried over with only a few modifications. 
First, the reproducing kernels remain the same, but the function 
\begin{math}f\end{math} is now \begin{math}\Cd\end{math}--valued and the 
integrals must be interpreted as a vector--valued integrals (that is, integrate 
in each coordinate). To make this more precise, if \begin{math}f\end{math} is
a \begin{math}\Cd\end{math}-valued function on \begin{math}\Bn\end{math}, and
$\{e_k\}_{k=1}^{d}$ is the standard orthonormal basis for $\Cd$,  
define:
\begin{align*}
\int_{\Bn}f(z)\dva (z) &:= \sum_{k=1}^{d} \int_{\Bn} 
\ip{f(z)}{e_{k}}_{\Cd}\dva (z)e_{k}.
\end{align*}

Let $L_{M^{d}}^{\infty}$ denote the set of $d\times d$ matrix--valued functions
such that the function \begin{math}z\mapsto\norm{\varphi(z)}\end{math} is an 
\begin{math}L^{\infty} \end{math} function. Note that it is not particularly 
important which matrix norm is used -- since \begin{math}\Cd\end{math} is finite
dimensional all norms are equivalent. The second change is that our symbols are 
now matrix--valued functions in $L_{M^{d}}^{\infty}$.

Define the Toeplitz algebra, denoted by $\mathcal{T}_{p,\alpha}$,
to be the SOT closure of finite sums of finite products of Toeplitz operators 
with $L_{M^{d}}^{\infty}$ symbols.

Finally, we change the way that we define the Berezin transform of an operator. 
The Berezin transform will 
be a matrix--valued function, acting on $\Cd$, given by the following relation 
(see also \cite{AE}):

\begin{align}
\ip{\widetilde{T}(z)e}{h}_{\Cd}=\ip{T(k_{z} e)}{k_{z} h}_{A_{\alpha}^{2}}
\end{align}
for \begin{math}e,h\in\Cd \end{math}.

We are now ready to state the main theorem of the paper, but first 
we need to introduce an auxiliary operator.
\begin{defn}
Let \begin{math}U_{z}\end{math} be defined by: 
\begin{math}(U_{z}f)=(f\circ \varphi_{z})  k_{z}^{(2,\alpha)}\end{math}. 
That is,  
\begin{align*}
(U_{z}f)(w)=(f\circ\varphi_{z})(w) k_{z}^{(2,\alpha)}(w).
\end{align*}
\end{defn}

\begin{defn} For \begin{math}T\in\mathcal{L}(\lpa)\end{math} set
\begin{math}T_{z}=U_{z}TU_{z}\end{math}.
\end{defn}

\begin{thm}\label{thm:main0}
Let \begin{math}T\end{math} be an operator in 
\begin{math}\mathcal{L}(A_{\alpha}^{2})\end{math} which can be 
written as 
\begin{align*}
T=\sum_{j=1}^{m}\prod_{k=1}^{m_{j}}T_{u_{j,k}},
\end{align*} 
where 
\begin{math}u_{j,k}\in
L_{M^{d}}^{\infty}\end{math}. Then the following are equivalent: 
\begin{itemize}
\item[(1)] $T$ is compact;
\item[(2)] $\Langle\widetilde{T}(z)e,h\Rangle_{\C^{d}} \to 0$ as 
$z \to \partial \Bn$ $e,h\in \C^{d}$ with $\|e\|_{2}=\|h\|_{2}=1$;
\item[(3)] $T_{z}e \to 0$ weakly as $z \to \partial \Bn$, 
    where $e\in\Cd$;
\item[(4)] $\|T_{z}e\|_{L_{\alpha}^{p}} \to 0$ as 
$z\to \partial \Bn$ for any $p>1$.
\end{itemize}
\end{thm}

Our main interest is the equivalence between $(1)$ and $(2)$ above, but it is
easier to prove their equivalence by proving that all four statements in 
Theorem \ref{thm:main0} are equivilant. 

\subsection{Discussion of the Theorem.}By now, there are many results that 
relate the compactness of an operator to its
Berezin transform. For example, if $T\in \mathcal{L}(A_{0}^{2}(\B;\C))$ 
can be written as a finite sum of finite products of Toeplitz operators,
Axler and Zheng prove in \cite{AZ} that $T$ is compact if and only if its 
Berezin transform vanishes on the boundary of $\B$. (Recall that 
$(A_{0}^{2}(\B;\C))$ is the standard unwieghted Bergman space on the unit
ball in $\C$). This was improved independently by 
Raimondo who extended the result to the spaces 
\begin{math}A_{0}^{2}(\Bn;\C)\end{math} in \cite{RAI} and Engli\v{s} who 
extended the results in great generality to Bergman spaces on bounded symmetric
domains in \cite{eng}. 

There are also several results along these lines for more general operators
than those that can be written as finite sums of finite products of 
Toeplitz operators. In \cite{eng1} Engli\v{s} proves that any compact
operator is in the operator--norm topology closure of the set of 
finite sums of finite products of Toeplitz operators (this is called the
Toeplitz algebra.) In \cite{Sua}, Su{\'a}rez proves that an operator, $T$, 
in $\mathcal{L}(A_{0}^{p}(\Bn;\C))$ is compact if and only if it is in the
Toeplitz algebra and its Berezin transform vanishes on $\partial\Bn$. This
was extended to the weighted Bergman spaces $(A_{\alpha}^{p}(\Bn;\C))$
in \cite{MSW} By Su{\'a}rez, Mitkovski, and Wick. 
Mitkovski and Wick achieve similar results for Bergman spaces on the polydisc
in \cite{mw1} and they extend these results to bounded symmetric domains 
in \cite{mw2}.

It is interesting to note that the hypothesis that $T$ is a finite sum of finite
products of Toeplitz operators is used only in Lemma \ref{lem:only} to obtain
an easy estimate. (This is also true for proofs of several of the results
mentioned above. See, for example, the proofs in \cite{AZ} and \cite{RAI}.) 
In particular, it is shown that:
\begin{align}\label{eqn:otherop}
\sup_{\|e\|_{2}=1}\sup_{z\in \Bn}\|T_{z}e\|_{A_{\alpha}^{p}} < 
\infty.
\end{align}
Therefore, instead of requiring that $T$ be a finite sum of finite products
of Toeplitz operators, we can require that $T$ satisfy \eqref{eqn:otherop}.
Specifically, as a corollary of the proof of Theorem \ref{thm:main0}, the
following holds:
\begin{cor}
Let $T$ be an operator in $\mathcal{L}(A_{\alpha}^{2})$ that satisfies
\eqref{eqn:otherop}. Then the following are equivalent:
\begin{itemize}
\item[(1)] $T$ is compact;
\item[(2)] $\Langle\widetilde{T}(z)e,h\Rangle_{\C^{d}} \to 0$ as 
$z \to \partial \Bn$ $e,h\in \C^{d}$ with $\|e\|_{2}=\|h\|_{2}=1$;
\item[(3)] $T_{z}e \to 0$ weakly as $z \to \partial \Bn$, 
    where $e\in\Cd$;
\item[(4)] $\|T_{z}e\|_{L_{\alpha}^{p}} \to 0$ as 
$z\to \partial \Bn$ for any $p>1$.
\end{itemize}
\end{cor}

This is not a new observation. Indeed, in \cite{mz}, the authors prove a 
similar result.

\section{Preliminaries}
We first fix notation that will last for the rest of the paper. The vectors 
$\{e_{i}\}$, etc. will denote the standard orthonormal basis 
vectors in $\Cd$. The letter $e$ will always denote a unit vector in $\Cd$. 
For vectors in $\Cd$, 
$\norm{\cdot}_{p}$ will denote the $l^{p}$ norm on $\Cd$. If 
$M$ is a $d\times d$ matrix, $\norm{M}$ will denote any convenient matrix norm. 
Since all 
norms of matrices are equivalent in finite dimensions, the exact norm used does 
not matter for quantitative considerations. Additionally, 
$M_{(i,j)}$ will denote the $(i,j)$ entry of $M$.
Finally, to lighten notation, fix an integer $d>1$, an integer $n\geq 1$ and a 
real $\alpha > -1$. Because of
this, we will usually suppress these constants in our notation. For example, the
reproducing kernels will be written as $K_{z}$ instead of $K_{z}^{(\alpha)}$
and $k_{z}^{(2,\alpha)}$ is simply $k_{z}$. Furthermore, we will write
$L_{\alpha}^{2}$ and $A_{\alpha}^{2}$ instead of $L_{\alpha}^{2}(\Bn;\Cd)$ and
$A_{\alpha}^{2}(\Bn;\Cd)$.
(We keep the $\alpha$ in the notation
for the spaces because this is customary). 

\subsection{Well-Known Results and Extensions to the Present Case}
We will discuss several well-known results about the standard Bergman Spaces, 
$A_{\alpha}^{p}(\Bn;\C)$ and state and prove their generalizations to the 
present vector-valued Bergman Spaces, $\apa$. 

Let $\varphi_{z}$ be the automorphisms of the ball that 
interchange $z$ and $0$. The automorphisms are used to define the following 
metrics:
\begin{align*}
\rho(z,w):=\abs{\varphi_{z}(w)} \quad\textnormal{and}\quad 
\beta(z,w):=\frac{1}{2}\log\frac{1+\rho(z,w)}{1-\rho(z,w)}.
\end{align*}

These metrics are invariant under the maps $\varphi_{z}$. Let $D(z,r)$ be the 
ball in the $\beta$ metric centered at $z$ with raduis $r$. Recall the following
identity:
\begin{align*}
 1-\abs{\varphi_{z}(w)}^{2}=\frac{(1-\abs{z}^{2})(1-\abs{w}^{2})}
 {\abs{1-\overline{z}w}^{2}}
\end{align*}
The following change of variables formula is in \cite[Prop 1.13]{Zhu}:
\begin{align}
\label{eqn:cov}
\int_{\Bn}f(w)\dva (w) = \int_{B_n}(f\circ \varphi_{z})(w) 
\abs{\nk{z}{2}(w)}^2 \dva (w). 
\end{align}
Straight-forward computations reveal:
\begin{align}
\label{eqn:eqOne}
\nk{z}{2}(\varphi_{z}(w)) \nk{z}{2}(w) \equiv 1.
\end{align}

The following propositions appear in \cite{Zhu}.

\begin{prop}\label{prop:soh2.20}
If \begin{math}a\in \Bn\end{math} and \begin{math}z\in D(a,r)\end{math}, 
there exists a constant depending only on \begin{math}r\end{math} such that 
\begin{math}1-|a|^2 \simeq 1-|z|^{2} \simeq \abs{1-\ip{a}{z}} \end{math}.
\end{prop}

\begin{prop}\label{prop:soh2.24}
Suppose $r>0$, $p>0$, and $\alpha > -1$. Then there exists a positive constant
that depends only on $\alpha$ and $r$ such
that
\begin{align*}
\abs{f(\lambda)}^{p} \lesssim \frac{1}{(1-|\lambda|^{2})^{n+1+\alpha}}
\int_{D(\lambda,r)}\abs{f(w)}^{p}dv_{\alpha}(w)
\end{align*}
for all holomorphic \begin{math}f:\Bn\to\C\end{math} and all 
\begin{math}\lambda\in \Bn\end{math}.
\end{prop}
\noindent The following vector-valued analogue will be used: 
\begin{prop}\label{prop:vvsoh2.24}
Suppose $r>0$, $p\geq 1$, and $\alpha>-1$. Then there exists a positive 
constant that depends only on $\alpha,r$, and $d$ such that
\begin{align*}
\sup_{z\in D(\lambda,r)}\norm{f(z)}_{p}^{p} \lesssim
\frac{1}{(1-|\lambda|^{2})^{n+1+\alpha}}
\int_{D(\lambda,2r)}\norm{f(w)}_{p}^{p}dv_{\alpha}(w)
\end{align*}
for all holomorphic \begin{math}f:\Bn\to\Cd\end{math} and all 
\begin{math}\lambda\in \Bn\end{math}
\end{prop}

\begin{proof}Let $q$ be conjugate exponent to $p$. Then
\begin{align*}
\sup_{z\in D(\lambda,r)}\|f(z)\|_{p}^{p} = 
\sup_{z\in D(\lambda,r)}\sup_{\|e\|_{q}=1}\abs{\ip{e}{f(z)}_{\Cd}}^{p}.
\end{align*}
By definition, \begin{math}\langle e,f(z)\rangle_{\Cd}\end{math} is 
holomorphic for 
all \begin{math}e\in \C^{d}\end{math}. By Proposition \ref{prop:soh2.24} and
Proposition \ref{prop:soh2.20}, for 
\begin{math}e\in\Cd\end{math} and \begin{math}z\in D(\lambda,r)\end{math} 
there holds:
\begin{align*}
|\langle e,f(z) \rangle_{\Cd} |^{p} &\lesssim \frac{1}{(1-|z|^{2})^{n+1+\alpha}}
\int_{D(z,r)}|
\langle e,f(w)\rangle_{\Cd}|^{p}dv_{\alpha}(w)
\\& \lesssim \frac{1}{(1-|z|^{2})^{n+1+\alpha}}\int_{D(z,r)}
   \|f(w)\|_{p}^{p}dv_{\alpha}(w)
\\& \simeq \frac{1}{(1-|\lambda|^{2})^{n+1+\alpha}}\int_{D(z,r)}
   \|f(w)\|_{p}^{p}dv_{\alpha}(w)
\\& \leq \frac{C}{(1-|\lambda|^{2})^{n+1+\alpha}}\int_{D(\lambda,2r)}
   \|f(w)\|_{p}^{p}dv_{\alpha}(w).
\end{align*}
Which completes the proof.
\end{proof}

The next lemma is in \cite{Zhu}:

\begin{lem}\label{lem:intEst}
\label{Growth}
For $z\in \Bn$, $s$ real and $t>-1$, let
\begin{align*}
F_{s,t}(z):=\int_{\Bn}\frac{(1-\abs{w}^2)^t}{\abs{1-\overline{w}z}^s}\,dv(w).
\end{align*}
Then $F_{s,t}$ is bounded if $s<n+1+t$ and grows as $(1-\abs{z}^2)^{n+1+t-s}$ 
when $\abs{z}\to 1$ if $s>n+1+t$.
\end{lem}

The next lemma is the version of Schur's Test that we will use. Note that it is 
essentially the same as the ``usual'' Schur's Test. The proof is omitted. 
\begin{lem}\label{eqn:vvschur}Suppose $(X,\mu)$ is a measure space,
$1<p<\infty$, and $q$ conjugate exponent to $p$.
Let \begin{math}T\end{math} be an integral operator with (matrix-valued) kernel 
\begin{math}M(x,y)\end{math} and let $f$ be vector--valued. That is, 
\begin{align*}
(Tf)(x)=\int_{X}M(x,y)f(y)d\mu(y).
\end{align*}
If there is a \begin{math}C_1\end{math} and a \begin{math}C_2\end{math} and a 
positive function \begin{math}h\end{math} such that the following is true:
\begin{align*}
\int_{X}\|M(x,y)\|h(y)^{q} d\mu(y) 
&\leq C_1h(x)^{q}
\end{align*}
for almost every $x\in X$, and
\begin{align*}
\int_{X}\|M(x,y)\|h(x)^{p} d\mu(x) 
&\leq C_2 h(y)^{p}
\end{align*}
for almost every $y\in X$
then \begin{math}T:L^{p}(X,\mu)\to L^{p}(X,\mu)\end{math} is bounded with 
norm at most \begin{math}C_{1}^{1/q}C_{2}^{1/p}\end{math}.
\end{lem}

There is also a vector--valued test to determine an operator's membership in the
Hilbert-Schmidt Class.
\begin{defn}
The matrix--valued function, $M(z,w)$, is in 
\begin{math}L^{2}(\Bn\times\Bn,\dva\times\dva)\end{math} if
\begin{align*}
\norm{M}_{L^{2}(\Bn\times\Bn,\dva \times \dva)}^{2}:=
\int_{\Bn}\int_{\Bn}\norm{M(z,w)}^2dv_{\alpha}(w)dv_{\alpha}(z)<\infty.
\end{align*}
\end{defn}
\begin{lem}[Vector--Valued Hilbert Schmidt Test]\label{lem:vvhs}
A linear operator \begin{math}T\in\mathcal{L}(L_{\alpha}^{2})\end{math} is 
Hilbert-Schmidt if there is a matrix-valued function \begin{math}M\end{math} in 
\begin{math}L^{2}(\Bn\times\Bn,\dva\times\dva)\end{math} such that
\begin{align*}
(Tf)(z)=\int_{\Bn}M(z,w)f(w)\dva (w).
\end{align*}
\end{lem}
\begin{proof}
Let $M^{*}(z,w)$ be the adjoint of the matrix $M(z,w)$.
Let \begin{math}\{\varphi_{n}\}\end{math} be an orthonormal basis for 
\begin{math}L_{\alpha}^{2}\end{math}, let \begin{math}e\in\Cd\end{math} with 
\begin{math}\|e\|_{2}=1\end{math} and let  \begin{math}M_{z}(w)=
M(z,w)\end{math}. Then there holds:
\begin{align*}
\|M_{z}^{*}e\|_{L_{\alpha}^{2}}^{2} 
&= \sum_{n=1}^{+\infty}\abs{\ip{M_{z}^{*}e}{\varphi_{n}}_{A_{\alpha}^{2}}}^{2}
\\&= \sum_{n=1}^{+\infty} \left|\int_{\Bn}\Langle M_{z}^{*}(w)e,\varphi_{n}(w)
\Rangle_{\C^{d}}\dva (w)\right|^{2}
\\&= \sum_{n=1}^{+\infty} \left|\int_{\Bn}
    \Langle e,M_{z}(w)\varphi_{n}(w)\Rangle_{\C^{d}}\dva (w)\right|^{2}
\\&= \sum_{n=1}^{+\infty} \left|
    \Langle e,\int_{\Bn}M_{z}(w)\varphi_{n}(w)\dva (w)
    \Rangle_{\C^{d}}\right|^{2}
\\&= \sum_{n=1}^{+\infty} \left|
    \Langle e,(T\varphi_{n})(z)\Rangle_{\C^{d}}\right|^{2}.
\end{align*}
Using this computation, there holds:
\begin{align*}
\sum_{n=1}^{+\infty}\|T\varphi_{n}\|_{L_{\alpha}^{2}}^{2} 
&= \sum_{n=1}^{+\infty}\int_{\Bn}\|(T\varphi_{n})(z)\|_{2}^{2}\dva (z)
\\&= \sum_{i=1}^{d}\int_{\Bn}\sum_{n=1}^{+\infty}\left|
    \Langle e_{i},(T\varphi_{n})(z)\Rangle_{\Cd}\right|^{2}\dva (z)
\\&= \sum_{i=1}^{d}\int_{\Bn}\|M_{z}^{*}e_{i}\|_{L_{\alpha}^{2}}^{2}\dva (z)
\\&= \sum_{i=1}^{d} \int_{\Bn}\int_{\Bn} 
    \|M^{*}(z,w)e_{i}\|_{2}^{2}\dva (w) \dva (z)
\\&\leq \sum_{i=1}^{d} \int_{\Bn}\int_{\Bn} 
    \|M(z,w)\|^{2}\dva (w) \dva (z)
\\&=d
< \infty.
\end{align*}
Therefore, $T$ is a Hilbert-Schmidt Operator.
\end{proof}

\begin{lem}\label{lem:uaunit}
For any \begin{math}z\in \Bn \text{, the operator } U_{z} \end{math} is
self-adjoint, idempotent, and isometric on \begin{math}A_{\alpha}^{2}\end{math}.
\end{lem}
\begin{proof}
We first show idempotency: For \begin{math}f \in A_{\alpha}^{2}\end{math}, 
\begin{align}
U^{2}_zf &= U_z(U_{z}f)\notag
\\&= ((U_zf)\circ\varphi_z) (k_{z}) \notag
\\&= ((f \circ \varphi_z)  (k_{z}\circ \varphi_z) (k_{z})\notag
\\&= (f \circ \varphi_z \circ \varphi_z) (k_{z}\circ\varphi_z) 
    (k_{z}) \label{eqn:uz}
\\&= f 1 \notag
\\&= f. \notag
\end{align}
The equality in \eqref{eqn:uz} follows from \eqref{eqn:eqOne}.
Next we show isometry: 
\begin{align}
\norm{U_{z}f}_{A_{\alpha}^{2}}^{2} 
&=\int_{\Bn} \norm{(U_{z}f)(w)}_{\Cd}^{2}\dva (w)\notag
\\&=\int_{\Bn} \sum_{k=1}^{d} 
    \abs{\ip{(U_{z}f)(w)}{e_{k}}_{\Cd}}^{2}\dva (w)\notag
\\&= \sum_{k=1}^{d}\int_{\Bn}\abs{\ip{(f\circ \varphi_{z})(w)k_{z}(w)}
    {e_{k}}_{\Cd}}^{2}\dva (w)\notag
\\&= \sum_{k=1}^{d}\int_{\Bn} \left|\Langle ((f\circ \varphi_{z})(w) ,
    e_{k}\Rangle_{\Cd}\right|^{2}\abs{k_{z}(w)}^{2}\dva (w)
    \notag
\\&= \sum_{k=1}^{d}\int_{\Bn}\abs{\ip{f(w)}{e_{k}}_{\Cd}}^{2}\dva (w)
    \label{eqn:iso2}
\\&=\|f\|_{A_{\alpha}^{2}}^{2}\notag,
\end{align} 
where 
\eqref{eqn:iso2} is due to the change of variables formula \eqref{eqn:cov}.
To show that \begin{math}U_{z}\end{math} is self-adjoint, 
we note the two previous conditions imply that:
\begin{align*}
\Langle U_{z}f,g\Rangle_{A_{\alpha}^{2}} 
= \Langle(U_{z}U_{z})f,U_{z}g\Rangle_{A_{\alpha}^{2}}
= \Langle f,U_{z}g\Rangle_{A_{\alpha}^{2}}.
\end{align*}
\end{proof}

\section{Basic Lemmas} 
We begin by investigating the relationship between the Berezin Transform and the
maps \begin{math}\varphi_{z}\end{math}. 

\begin{lem}If \begin{math}T\in\mathcal{L}(L_{\alpha}^{2})\end{math} 
and \begin{math}z\in \Bn\end{math}, then
\begin{math}\widetilde{T}\circ\varphi_{z} = \widetilde{T_{z}}\end{math}.
\end{lem}
\begin{proof}
Suppose that \begin{math}T\in \mathcal{L}(L_{\alpha}^{2})
\end{math} and \begin{math}z,w\in \Cn\end{math}, then there holds
\begin{align*}
U_{z}(K_{w}e)=(\overline{k_{z}(w))} 
(K_{\varphi(w)}^{(\alpha)})e.
\end{align*}
Indeed, if \begin{math} f 
\end{math} is in \begin{math}L_{\alpha}^{2}\end{math} then there holds, for 
\begin{math}\norm{e}\in\Cd\end{math}:
\begin{align*}
\ip{f}{U_{z}(K_{w}e)}_{A_{\alpha}^2} 
&=\ip{U_{z}f}{K_{w}e}_{A_{\alpha}^2}
\\&=\ip{(U_{z}f)(w)}{e}_{\C^{d}}
\\&=\ip{(f\circ\varphi_{z})(w) \times k_{z}(w)}{e}_{\C^{d}}
\\&=\ip{(f\circ\varphi_{z})(w)}{\overline{k_{z}(w)}e}_{\C^{d}}
\\&=\ip{f}{\overline{k_{z}(w)}
    {K_{\varphi_{z}(w)}e}}_{A_{\alpha}^{2}},
\end{align*}
which implies the claim. 
By writing this equality as:
\begin{align*}
U_{z}(K_{w}e) 
&= \norm{K_{\varphi_{z}(w)}}_{A_{\alpha}^{2}}
    \norm {K_{\varphi_{z}(w)}}_{A_{\alpha}^{2}}^{-1}(U_{z}
    K_{w}e)
\\&= \norm {K_{\varphi_{z}(w)}}_{A_{\alpha}^{2}}
    \overline{k_{z}(w)}
    \lp\norm{K_{\varphi_{z}(w)}}_{A_{\alpha}^{2}}^{-1}
    K_{\varphi_{z}(w)}\rp e.
\end{align*}
It follows that
\begin{align}\label{eqn:equivU}
U_{z}(K_{w}e)
=\left(\norm{K_{w}}_{A_{\alpha}^{2}}^{-1}
    \norm{K_{\varphi_{z}(w)}}_{A_{\alpha}^{2}}
    \overline{K_{z}(w)}\right)
    \left(k_{\varphi_{z}(w)}\right)e.
\end{align}
Since $U_{z}$ is an isometry, the $A_{\alpha}^{2}$ norm of 
$K_{w}=\norm{K_{w}}_{A_{\alpha}^{2}}k_w$ is equal to the $A_{\alpha}^{2}$ 
norm of
the extreme right side of
of \eqref{eqn:equivU}. Computing this norm and using the fact that 
\begin{math}\left\langle k_{w}e,k_{w}e\right
\rangle_{A_{\alpha}^{2}}=
\left\langle k_{\varphi_{z}(w)}e,k_{\varphi_{z}(w)}e
\right\rangle_{A_{\alpha}^{2}} = 
1\end{math}, gives:
\begin{align*}
\norm{K_{w}}_{A_{\alpha}^{2}}^{-1}
    \norm{K_{\varphi_{z}(w)}}_{A_{\alpha}^{2}}
    \abs{\overline{K_{z}(w)}}
=\norm{K_{w}}_{A_{\alpha}^{2}}.
\end{align*}
Therefore there holds:
\begin{align}\label{eqn:uz2}
\norm{K_{w}}_{A_{\alpha}^{2}}^{-1}
    \norm{K_{\varphi_{z}(w)}}_{A_{\alpha}^{2}}
    \abs{\overline{k_{z}(w)}}
=1.
\end{align} 
We use this fact to make the following computation,
for \begin{math}\|e_{1}\|_{2}=\|e_{2}\|_{2}=1\end{math}:
\begin{align}\label{eqn:uz3}
\ip{\widetilde{T}\circ\varphi_{z}(w)e_{1}}{e_{2}}_{\C^{d}}
&= \ip{\widetilde{T}(\varphi_{z}(w))e_{1}}{e_{2}}_{\C^{d}}\notag
\\&=\ip{T(k_{\varphi_{z}(w)}e_{1})}
   {k_{\varphi_{z}(w)}e_{2}}_{A_{\alpha}^{2}}\notag
\\&=\left(\norm{K_{w}}_{A_{\alpha}^{2}}^{-1}
    \norm{K_{\varphi_{z}(w)}}_{A_{\alpha}^{2}}
    \abs{K_{z}(w)}\right)^{-2}  
    \ip{TU_{z}(K_{w}e_{1})}{U_{z}
    (K_{w}e_{2})}_{A_{\alpha}^{2}}\notag
\\&=\left(\norm{K_{w}}_{A_{\alpha}^{2}}^{-2}
    \norm{K_{\varphi_{z}(w)}}_{A_{\alpha}^{2}}
    \abs{K_{w}(z)}\right)^{-2}
    \ip{TU_{z}(k_{w}e_{1})}{U_{z}
    (k_{w}e_{2})}_{A_{\alpha}^{2}}\notag
\\&=\left(\norm{K_{w}}_{A_{\alpha}^{2}}^{-1}
    \norm{K_{\varphi_{z}(w)}}_{A_{\alpha}^{2}}
    \abs{k_{w}(z)}\right)^{-2}
    \ip{TU_{z}(k_{w}e_{1})}{U_{z}
    (k_{w}e_{2})}_{A_{\alpha}^{2}}\notag
\\&=\ip{TU_{z}(k_{w}e_{1})}{U_{z}
    (k_{w}e_{2})}_{A_{\alpha}^{2}}
\\&=\ip{U_{z}TU_{z}(k_{w}e_{1})}
    {k_{w}e_{2}}_{A_{\alpha}^{2}}\notag
\\&=\ip{T_{z}(k_{w}e_{1})}{k_{w}
    e_{2}}_{A_{\alpha}^{2}}\notag
\\&=\ip{\widetilde{T_{z}}e_{1}}{e_{2}}_{\C^{d}}.\notag
\end{align}
In the equality in \eqref{eqn:uz3} we used \eqref{eqn:uz2}.
This shows that \begin{math}\widetilde{T}\circ\varphi_{z}=\widetilde{T_{z}}
\end{math} as maps on
\begin{math}\cd\end{math}.
\end{proof}
\begin{lem}\label{lem:rai3.2}
For every \begin{math}u\in L_{M_{d}}^{\infty}\end{math} 
and for every \begin{math}z\in \Bn\end{math}, there holds: 
\begin{align*}
U_{z}T_{u}U_{z}=T_{u\circ\varphi_{z}}
\end{align*}.
\end{lem}
\begin{proof}
Since \begin{math}U_{z}\end{math} is idempotent, it is enough to 
prove that \begin{math}T_{u}U_{z}=
U_{z}T_{u\circ\varphi_{z}}\end{math}. To this end, we compute 
\begin{math}T_{u}U_{z}\end{math}. Let 
\begin{math}f\end{math} be in \begin{math}A_{\alpha}^{2}\end{math}
and \begin{math}e\in\Cd\end{math}, then:
\begin{align*}
\ip{T_{u}U_{z}f}{k_{w}e}_{A_{\alpha}^{2}} 
&=\ip{T_{u} \lp\lp f\circ\varphi_{z}\rp k_{z}\rp}
    {k_{w}e}_{A_{\alpha}^{2}}
\\&=\ip{PM_{u}\lp\lp f\circ\varphi_{z}\rp k_{z}\rp}
    {k_{w}e}_{A_{\alpha}^{2}}
\\&= \ip{ u\lp f\circ\varphi_{z}\rp k_{z}}
    {k_{w}e}_{A_{\alpha}^{2}}
\\&=\int_{\Bn}\ip{u(\eta)\lp f\circ\varphi_{z}\rp(\eta)k_{z}(\eta)}
    {k_{w}(\eta)e}_{\C^{d}}\dva (\eta).
\end{align*}
Now, we calculate \begin{math}U_{z}T_{u\circ\varphi_{z}}\end{math}. 
For \begin{math}e\in\Cd\end{math}:
\begin{align*}
\ip{U_{z}T_{u\circ\varphi_{z}}f}{k_{w} e}_{A_{\alpha}^{2}} 
&=\ip{T_{u\circ\varphi_{z}}f}{U_{z}
    \lp k_{w}e\rp}_{A_{\alpha}^{2}}
\\&=\ip{PM_{u\circ\varphi_{z}}f}{U_{z}\lp 
    k_{w}e\rp}_{A_{\alpha}^{2}} 
\\&= \ip{\lp u\circ\varphi_{z}\rp f}{\lp 
    k_{w}\circ\varphi_{z}\rp k_{z} e}_{A_{\alpha}^{2}}
\\&=\int_{\Bn} \ip{\lp u\circ\varphi_{z}\rp\lp\eta\rp 
    f\lp\eta\rp}{k_{z}\lp\eta\rp\lp 
    k_{w}\circ\varphi_{z}\rp\lp\eta\rp e}_{\C^{d}} \dva\lp\eta\rp 
    =: \mathcal{A}.
\end{align*}

Make the substitution \begin{math}\eta = \varphi_{z}\lp\xi\rp\end{math}. 
Then there holds, for \begin{math}e\in\Cd\end{math}, using Lemma \ref{eqn:cov},
and the fact that $k_{z}\circ\varphi_{z}(\xi)k_{z}(\xi)
\equiv 1$, 
\begin{align*}
\mathcal{A} 
&=\int_{\Bn} \ip{u(\xi)(f\circ\varphi_{z})(\xi)}
    {(k_{z}\circ\varphi_{z})(\xi)k_{w}(\xi)e}_{\C^{d}}
    \abs{k_{z}(\xi)}^{2}\dva (\xi)
\\&= \int_{\Bn} \ip{u(\xi)(f\circ\varphi_{z})(\xi)}
    {(k_{z}\circ\varphi_{z})(\xi)k_{z}(\xi)
    k_{w}(\xi) \overline{k_{z}(\xi)}e}_{\C^{d}}\dva (\xi)
\\&= \int_{\Bn} \ip{u(\xi)(f\circ\varphi_{z})(\xi)}
    {k_{w}(\xi)\overline{k_{z}(\xi)}e}_{\C^{d}}\dva (\xi)
\\&= \int_{\Bn} \ip{k_{z}(\xi)u(\xi)(f\circ\varphi_{z})(\xi)}
    {k_{w}(\xi)e}_{\C^{d}}\dva (\xi).
\end{align*}
This gives, \begin{math} 
\ip{T_{u}U_{z}f}{k_{w}e}_{A_{\alpha}^{2}} 
=\ip{U_{z}T_{u\circ\varphi_{z}}f}
{k_{w}e}_{A_{\alpha}^{2}}\end{math} 
for every \begin{math}w\in \Bn\end{math}
and \begin{math}e\in\Cd\end{math}. This completes the proof. 
\end{proof}

Before going on, we introduce a new operator on \begin{math}A_{\alpha}^{2}
\end{math}: 
\begin{math}(U_{\mathcal{R}}f)(w)=f(-w)\end{math}. Let
\begin{align*}
\mathcal{J}_{c,t}(z)
=\int_{\Bn}\frac{(1-|w|)^{t}\dva (w)}{|1-z\overline{w}|^{n+1+t+c}}.
\end{align*}
By Lemma \ref{lem:intEst}, for \begin{math}c < 0\end{math} and 
\begin{math}t > -1\end{math}, the 
function \begin{math}\mathcal{J}_{c,t}\end{math} is bounded on 
\begin{math}\Bn\end{math}. 
We will state a proposition that will be use in conjunction with Schur's Test 
later on. The proof can be easily deduced from the proof in \cite{RAI} and is 
omitted.  
\begin{lem}\label{lem:raiProp3.3}
Given \begin{math}p\in \R \text{ with }0<p-1<(n+1)^{-1}, \text{ and }T\in 
\mathcal{L}(A_{\alpha}^{2})
\text{ and }e\in\Cd\end{math}, 
then
\begin{align*}
\int_{\Bn} 
    \norm{U_{\mathcal{R}}TU_{\mathcal{R}}(K_{z}e)(w)}_{2}
\norm{K_{w}}_{A_{\alpha}^{2}}^{\epsilon}\dva (w)
&\leq \norm{K_{z}}_{A_{\alpha}^{2}}^{\epsilon}\lp\sup_{z\in \Bn}
    \norm{T_{-z}e}_{A_{\alpha}^{q}}\rp 
\sup_{z\in \Bn}
|\mathcal{J}_{a,b}(z)|^{1/p}
\\ \int_{\Bn} \|U_{\mathcal{R}}TU_{\mathcal{R}}
(K_{z}e)(w)\|_{2}\|K_{w}
\|_{A_{\alpha}^{2}}^{\epsilon}\dva (z)
&\leq\|K_{z}\|_{A_{\alpha}^{2}}^{\epsilon}
\lp\sup_{w\in \Bn}\|T_{-w}e\|_{A_{\alpha}^{q}}\rp 
\sup_{w\in \Bn}
|\mathcal{J}_{a,b}(w)|^{1/p}
\end{align*}
where: \begin{math}2(p-1)/p < \epsilon < 2/(n+1)2, 
\text{ }a=(p-1)(n+1)-(n+1)\epsilon p/2 \text{ and }
b=-(n+1)\epsilon p/2 \text{ and }p^{-1}+q^{-1}=1\end{math}. Moreover, the 
quantity, 
\begin{align*}
\sup_{z\in \Bn}|\mathcal{J}_{a,b}(z)|^{1/p}
\end{align*}
is finite. 
\end{lem}
\begin{lem} \label{lem:only}
Let \begin{math}T\end{math} be an operator in 
\begin{math}\mathcal{L}(A_{\alpha}^{2})\end{math} which can be 
written as \begin{math}T=\sum_{j=1}^{m}\Pi_{k=1}^{m_{j}}T_{u_{j,k}}\end{math}, 
where \begin{math}u_{j,k}\in L_{M^{d}}^{\infty}\end{math}. Then, for every 
\begin{math}1<p<\infty, \text{ }
\sup_{\|e\|_{2}=1}\sup_{z\in \Bn}\|T_{z}e\|_{A_{\alpha}^{p}} < 
\infty\end{math}.
\end{lem}

\begin{proof}
We can assume that \begin{math}T=\Pi_{k=1}^{m}T_{u_{j}}\end{math}. Using Lemma 
\ref{lem:rai3.2}, we have that 
\begin{math}T_{z}=\Pi_{k=1}^{m}T_{u_{j}\circ\varphi_{z}}\end{math}. 
Since \begin{math}P_{\alpha}\end{math} is 
bounded from \begin{math}L_{p}^{\alpha}\to A_{p}^{\alpha}\end{math}, we have 
\begin{math}\|P_{\alpha}f\|_{A_{p}^{\alpha}}\lesssim 
\|f\|_{L_{\alpha}^{p}}\end{math}. Therefore, 
since \begin{math}\|u\circ\varphi_{z}\|_{\infty}=\|u\|_{\infty}
\end{math}, we have \begin{math}
\|T_{u\circ\varphi_{z}}f\|_{L_{\alpha}^2}\lesssim
\|u\|_{\infty}\|f\|_{L_{\alpha}^2}\end{math}.
This implies 
that \begin{math} \|T_{z}e\|_{L_{\alpha}^p}\leq \Pi_{k=1}^{m}
\|T_{u_{k}\circ\varphi_{z}}e\|_{L_{\alpha}^p} 
\lesssim \Pi_{k=1}^{m}\|u_{k}\|_{\infty}\end{math}. Since the right hand side of
this 
estimate is independent of \begin{math}z\end{math} (the implied constant depends
only on \begin{math}p\end{math}), we are done. 
\end{proof}

\section{The Main Theorem} 
For convenience, we remind the reader of the main theorem. 
\begin{thm}
Let \begin{math}T\end{math} be an operator in 
\begin{math}\mathcal{L}(A_{\alpha}^{2})\end{math} which can be 
written as 
\begin{align*}
T=\sum_{j=1}^{m}\prod_{k=1}^{m_{j}}T_{u_{j,k}},
\end{align*} 
where 
\begin{math}u_{j,k}\in
L_{M^{d}}^{\infty}\end{math}. Then the following are equivalent: 
\begin{itemize}
\item[(1)] $T$ is compact;
\item[(2)] $\Langle\widetilde{T}(z)e,h\Rangle_{\C^{d}} \to 0$ as 
$z \to \partial \Bn$ $e,h\in \C^{d}$ with $\|e\|_{2}=\|h\|_{2}=1$;
\item[(3)] $T_{z}e \to 0$ weakly as $z \to \partial \Bn$;
\item[(4)] $\|T_{z}e\|_{L_{\alpha}^{p}} \to 0$ as 
$z\to \partial \Bn$ for any $p>1$.
\end{itemize}
\end{thm}
\begin{proof}\begin{math}(1)\implies (2)\end{math}.
First, suppose that \begin{math}T\end{math} is compact. Observe that 
\begin{math}k_{z}e\to 0\end{math} weakly in 
\begin{math}L_{\alpha}^{2}\end{math} as \begin{math}z\to \partial \Bn\end{math}.
Indeed, if \begin{math}f\in L_{\alpha}^{2}\end{math} then 
\begin{align*}\left|\Langle f,\kz e\Rangle_{\ata}\right|&\leq
\sum_{k=1}^{d} \left|\overline{\Langle e,e_{k}\Rangle_{\cd}}\right|
\left|\Langle f,\kz e_{k}\Rangle_{\ata}\right|
\end{align*}
which goes to zero as \begin{math}z\to \partial\Bn\end{math}. (Here we used
the fact that $k_{z}\to 0$ weakly as $z\to\partial\Bn$ in 
$L_{\alpha}^{2}(\Bn;\C)$).
Since \begin{math}T\end{math} is compact, a well-known 
result about compact operators implies that \begin{math}Tk_{z}e\to 
0\end{math} strongly, that is \begin{math}
\|Tk_{z}e\|_{A_{\alpha}^{2}}\to 0 \text{ as }z\to \partial 
\Bn\end{math}. 
Then by the Cauchy-Schwarz inequality, there holds \begin{math}
\left|\Langle\widetilde{T}(z)e,h\Rangle_{\ata}\right|=
\left|\Langle Tk_{z}e,\kz h\Rangle_{A_{\alpha}^{2}}\right|
\leq \|Tk_{z}e\|_{A_{\alpha}^{2}}\to 0\end{math}
as $z\to\partial\Bn$. This gives 
\begin{math}(2)\end{math}.
\end{proof}
\begin{proof}\begin{math}(2)\implies (3)\end{math}.
If \begin{math}\{f_{k}\}\end{math} is a countable orthonormal basis for 
\begin{math}A_{\alpha}^{2}(\Bn;\C)\end{math} and 
\begin{math}\{e_{i}\}\end{math} is an orthonormal basis for 
\begin{math}\C^{d}\end{math}, then 
\begin{math}\{f_{k}e_{i}\}_{k,i}\end{math} is a countable orthonormal basis for 
\begin{math}A_{\alpha}^{2}(\Bn;\Cd)\end{math}.
Let \begin{math}\beta=(\beta_1,\ldots,\beta_n)\end{math} be a multi-index and
define \begin{math}p_{\beta}(\lambda)=
\lambda^{\beta}=(\lambda_{1}^{\beta_{1}},\ldots , 
\lambda_{n}^{\beta_{n}})\end{math}. Since this is an orthonormal basis for 
\begin{math}A_{\alpha}^{2}(\Bn;\Cd)\end{math}
(up to a normalization constant which does not matter to us) 
it is enough to show that \begin{math}\Langle T_{z}e,
p_{\beta}h \Rangle_{A_{\alpha}^{2}} \to 0\end{math} as \begin{math}z\to 
\Bn\end{math} for every 
\begin{math}\beta\in\N^{n}\end{math} and for every 
\begin{math}\|h\|_{2}=1\end{math}. 
We begin by observing that
\begin{math}
\Langle \widetilde{T}(\varphi_{z}(\lambda))e,h\Rangle_{\C^{d}} 
= \Langle\widetilde{T_{z}}(\lambda)e,h\Rangle_{\C^{d}} 
= \Langle T_{z}(k_{\lambda}e),k_{\lambda}h
\Rangle_{A_{\alpha}^{2}} \end{math}.
Expanding \begin{math}k_{\lambda}\end{math} in the orthonormal basis 
\begin{math}\{p_{\beta}\}\end{math}
we next observe that
\begin{align*}k_{\lambda}=
(1-|\lambda|^{2})^{(n+1+\alpha)/2}\sum_{\beta\in\N^{n}}
(\overline{\lambda^{\beta}}p_{\beta})/\gamma_{\beta}.
\end{align*}
Combining these two observations, we deduce that 
\begin{align*}
\Langle (\widetilde{T}\circ\varphi_{z})(\lambda)e,h\Rangle_{\C^{d}} 
&= \Langle T_{z}(k_{\lambda}e),k_{\lambda}h
\Rangle_{A_{\alpha}^{2}}
\\&= \Langle (T_{z})(1-|\lambda|^{2})^{(n+1+\alpha)/2}
    \sum_{\beta\in\N^{n}}
    (\overline{\lambda^{\beta}}p_{\beta}e)/\gamma_{\beta},
    \sum_{\tau\in\N^{n}}(\overline{\lambda^{\tau}}p_{l}h)/
    \gamma_{\tau}\Rangle_{A_{\alpha}^{2}}
\\&= (1-|\lambda|^{2})^{n+1+\alpha}\sum_{\beta,\tau \in \N^{n}}
\Langle T_{z}
    \frac{\overline{\lambda^{\beta}}p_{\beta}e}{\gamma_{\beta}},
    \frac{\overline{\lambda^{\tau}}
    p_{\tau}h}{\gamma_{\tau}}\Rangle_{A_{\alpha}^{2}}
\\&= (1-|\lambda|^{2})^{n+1+\alpha}
    \sum_{\beta,\tau\in\N^{n}}
    \frac{\Langle T_{z}(p_{\beta}e),(p_{\tau}h)
    \Rangle_{A_{\alpha}^{2}}}
    {\gamma_{\beta}{\overline{\gamma_{\tau}}}}
    \overline{\lambda^{\beta}}\lambda^{\tau}.
\end{align*}
Now, we multiply both sides of this equation by 
\begin{math}\frac{\overline{p_{\eta}
(\lambda)}}{k_{\lambda}(\lambda)}
\end{math} and integrate over \begin{math}r\Bn\end{math} 
in the variable \begin{math}\lambda\end{math}:
\begin{align*}
\int_{r\Bn}\Langle \widetilde{T}(\varphi_{z}(\lambda))
\frac{\overline{p_{\eta}(\lambda)}}
{k_{\lambda}(\lambda)}e,h\Rangle_{\C^{d}}\dva (\lambda)
&= \sum_{\beta,\tau\in\N^{n}}\frac{\Langle T_{z}(p_{\beta}e),
(p_{\tau}h)\Rangle_{A_{\alpha}^{2}}}
{\gamma_{\beta}{\overline{\gamma_{\tau}}}}
\int_{r\Bn}\overline{p_{\eta+\beta}(\lambda)}p_{\tau}(\lambda)\dva (\lambda).
\end{align*}
Computing the integral on the right hand side, and re-writing the left hand side
gives:
\begin{align*}
\int_{r\Bn}\overline{p_{\eta}(\lambda)}
(k_{\lambda})^{-1}(\lambda)\Langle
(\widetilde{T}\circ\varphi_{z})(\lambda)e,
h\Rangle_{\C^{d}} \dva (\lambda)
&=r^{2|\eta|+2}\sum_{\beta\in\N^{n}}\frac{\Langle T_{z}(p_{\beta}e),
p_{\eta +\beta}h
\Rangle_{A_{\alpha}^{2}}}
{\gamma_{\beta}}r^{2|\beta|}.
\end{align*}
Now, since \begin{math}(2)\end{math} holds, the left hand side goes to zero as 
\begin{math}z\to\partial \Bn\end{math} for fixed \begin{math}
0< r<\infty\end{math}. Divide both sides by 
\begin{math}r^{2|\eta|+2}\end{math}. This means the right hand side becomes:
\begin{align*}
\frac{\Langle T_{z}e,p_{\eta}h\Rangle_{A_{2}^{\alpha}}}{\gamma_{0}} +
\sum_{\beta\in \N^{n}\setminus 0}\frac{\Langle T_{z}(p_{\beta}e),
p_{\eta+\beta}h\Rangle_{A_{2}^{\alpha}}}
{\gamma_{\beta}}r^{2|\beta|}.
\end{align*}
Thus, we conclude that for fixed \begin{math}r\in (0,1)\end{math} we have 
\begin{align*}
\frac{\Langle T_{z}e,p_{\eta}h\Rangle_{A_{2}^{\alpha}}}{\gamma_{0}} +
\sum_{\beta\in \N^{n}\setminus0}\frac{\Langle T_{z}(p_{\beta}e),
p_{\eta+\beta}h\Rangle_{A_{2}^{\alpha}}}
{\gamma_{\beta}}r^{2|\beta|} \to 0.
\end{align*}
as \begin{math}z\to \partial \Bn\end{math}. Easy estimates also give: 
\begin{align*}
\left|
\sum_{\beta\in \N^{n}\setminus 0}\frac{\Langle T_{z}(p_{\beta}e),
p_{\eta+\beta}h\Rangle_{A_{2}^{\alpha}}}
{\gamma_{\beta}}r^{2|\beta|} \right| \leq
\|T\|\lp\sum_{|\beta|=1}^{n}r^{2|\beta|}+
\sum_{|\beta|>n}r^{2|\beta|}\rp.
\end{align*}
Observe two things. First, the quantity
\begin{align*}
\frac{\Langle T_{z}e,p_{\eta}h\Rangle_{A_{2}^{\alpha}}}{\gamma_{0}}
\end{align*}
is independent of \begin{math}r\end{math}, and second the quantity: 
\begin{align*}
\|T\|_{\mathcal{L}(\lpa)}\lp\sum_{|\beta|=1}^{n}r^{2|\beta|}
+\sum_{|\beta|>n}r^{2|\beta|}
\rp
\end{align*}
is the exact same quantity as in the proof the corresponding theorem 
in \cite{RAI}, where it is proven that for 
\begin{math}r\end{math} small enough, this quantity can be made 
smaller than \begin{math}\epsilon\end{math}
and thus 
\begin{align*}
\limsup_{z\to\partial \Bn}\left|\Langle T_{z}e,p_{\eta}h
\Rangle_{A_{\alpha}^{2}}\right| < \epsilon.
\end{align*}
Since this is true for every \begin{math}\epsilon\end{math} we 
conclude that \begin{math}\left|\Langle 
T_{z}e,p_{\eta}h\Rangle_{A_{\alpha}^{2}}\right|
\to 0\end{math} as 
\begin{math}z\to\partial \Bn\end{math}. Since 
\begin{math}\eta\end{math} and \begin{math}h\end{math} are arbitrary,
our claim has been proven. 
\end{proof}

\begin{proof}(3)\begin{math}\implies (4)\end{math} 
We need to show that if \begin{math}T_{z}e \to 0 \end{math} weakly as 
\begin{math}z \to \partial \Bn\end{math}, 
then 
\begin{math}\|T_{z}e\|_{L_{\alpha}^{p}} \to 0 \end{math} as 
\begin{math} z\to \partial \Bn \end{math} for any \begin{math}
1<p<\infty\end{math}. If \begin{math}r\in
(0,1)\end{math}, there holds:
\begin{align*}
\|T_{z}e\|_{L_{\alpha}^{2}} 
&= \int_{\Bn\setminus r\Bn}\Langle (T_{z}e)(w),
    (T_{z}e)(w)\Rangle_{\C^{d}}\dva (w)
    +\int_{r\Bn}\Langle (T_{z}e)(w),(T_{z}e)(w)\Rangle_{\C^{d}}\dva (w) 
\\&\leq \nu(\Bn\setminus r\Bn)^{1/2}\|T_{z}e\|_{L_{\alpha}^{4}} +
    \int_{r\Bn}\Langle (T_{z}e)(w),(T_{z}e)(w)\Rangle_{\C^{d}}\dva (w). 
\end{align*}
Choose \begin{math}r\end{math} close enough to 1 so that the first term on the 
right hand side smaller than any \begin{math}\delta > 0\end{math} 
(this is possible since 
\begin{math}\|T_{z}e\|_{L_{\alpha}^{4}}\end{math} is bounded independent of 
\begin{math}z\end{math}, by Lemma \ref{lem:only}). 
Now, a sequence of holomorphic
functions which converges weakly to zero also converges to zero in norm on 
compact sets. Thus, the second term on the right hand side goes to zero as 
\begin{math}z\to\partial\Bn\end{math}. This proves our claim for the case 
\begin{math}p=2\end{math}. We now assume that 
\begin{math}p\in (2,\infty)\end{math}. We have that:
\begin{align*}
\|T_{z}e\|_{L_{\alpha}^{2}}\leq\|T_{z}e\|_{L_{\alpha}^{2}}^{1/p}
\|T_{z}e\|_{L_{\alpha}^{2p-p}}^{(p-1)/p}.
\end{align*}
Again, since \begin{math}\|T_{z}\|_{A_{\alpha}^{2p-p}}\end{math} is bounded 
independent of \begin{math}z\end{math} and since 
\begin{math}\|T_{z}\|_{L_{\alpha}^{2}}\end{math} converges to zero as 
\begin{math}z\to\partial \Bn
\end{math} we have proven the claim for the case 
\begin{math}p\in (2,\infty)\end{math}. If \begin{math}p\in (1,2)
\end{math} we have that \begin{math}\|T_{z}e\|_{L_{\alpha}^{p}}\lesssim
\|T_{z}e\|_{L_{\alpha}^{2}}
\end{math}. This completes the proof. \end{proof}
\begin{proof}\begin{math}(4)\implies (1)\end{math}
Suppose that \begin{math}\|T_{z}e\|_{L_{\alpha}^{q}}\to 0 \end{math} as 
\begin{math} z\to\partial \Bn\end{math} for 
every \begin{math}q\in (1,\infty)\end{math}. Since 
\begin{math}U_{\mathcal{R}}\end{math} is bounded and invertible, we have that
\begin{math}T\end{math} is compact if and only if 
\begin{math}U_{\mathcal{R}}TU_{\mathcal{R}}=:T_{\mathcal{R}}
\end{math}
is compact. So, we will show that \begin{math}T_{\mathcal{R}}\end{math} is 
compact. To do this, we will show that \begin{math}T_{\mathcal{R}}\end{math} is 
an integral operator with matrix--valued kernel \begin{math}M(z,w)\end{math} 
given by the following relation 
\begin{align*}
\Langle M(z,w)e_{i},e_{j}\Rangle_{\C^{d}} = 
\Langle (T_{\mathcal{R}}K_{z}e_{j})(w),e_{i}\Rangle_{\C^{d}}.
\end{align*}
We then study radial truncations of the kernel, use Lemma \ref{lem:vvhs} to 
prove the truncations induce compact operators, and then make a limiting 
argument to show that \begin{math}T_{\mathcal{R}}\end{math} is compact.

Let \begin{math}e \in \C^{d}\end{math}. 
First, there holds:

\begin{align*}
(T_{\mathcal{R}}^{*}K_{w}e)(z)
&=\int_{\Bn}(T_{\mathcal{R}}^{*}K_{w}e)(\lambda)
    (\overline{K_{z}(\lambda)})\dva (\lambda)
\\&=\int_{\Bn}\sum_{i=1}^{d}
    \Langle (T_{\mathcal{R}}^{*}K_{w}e)(\lambda)(\overline{K_{z}(\lambda)}      
    ),e_{i}\Rangle_{\C^{d}}e_{i}\dva (\lambda)
\\&=\sum_{i=1}^{d}\int_{\Bn}
    \Langle (T_{\mathcal{R}}^{*}K_{w}e)(\lambda),
    K_{z}(\lambda)(e_{i})\Rangle_{\C^{d}}\dva (\lambda)e_{i}
\\&=\sum_{i=1}^{d}
    \Langle (T_{\mathcal{R}}^{*}K_{w}e),
    K_{z}e_{i}\Rangle_{A_{\alpha}^{2}}e_{i}
\\&=\sum_{i=1}^{d}\Langle K_{w}e,
    (T_{\mathcal{R}}K_{z}e_{i})\Rangle_{A_{\alpha}^{2}}e_{i}
\\&=\sum_{i=1}^{d}
    \overline{\Langle (T_{\mathcal{R}}K_{z}e_{i}),
    K_{w}e\Rangle_{A_{\alpha}^{2}}}e_{i}
\\&=\sum_{i=1}^{d}\overline{
    \Langle (T_{\mathcal{R}}K_{z}e_{i})(w),e\Rangle_{\C^{d}}}e_{i}
\\&=\sum_{i=1}^{d}\Langle e,
    (T_{\mathcal{R}}K_{z}e_{i})(w)\Rangle_{\C^{d}}e_{i}.
\end{align*}
By uniqueness of expansion in orthonormal bases, this implies that:
\begin{align*}
\Langle (T_{\mathcal{R}}^{*}K_{w}e)(z),e_{i}\Rangle_{\C^{d}} &=
\Langle e,(T_{\mathcal{R}}K_{z}e_{i})(w)\Rangle_{\C^{d}}.
\end{align*}
This computation yields:
\begin{align*}
\Langle (T_{\mathcal{R}}f)(w), e\Rangle_{\C^{d}}
&= \Langle T_{\mathcal{R}}f, K_{w}e\Rangle_{A_{\alpha}^{2}}
\\&= \Langle f,(T_{\mathcal{R}}^{*}K_{w}e)\Rangle_{A_{\alpha}^{2}}
\\&=\int_{\Bn}\Langle f(z),(T_{\mathcal{R}}^{*}K_{w}e)(z)
\Rangle_{\C^{d}}\dva (z)
\\&=\int_{\Bn}\sum_{j=1}^{d}\Langle f(z),e_{j}\Rangle_{\C^{d}}
\Langle e_{j},(T_{\mathcal{R}}^{*}K_{w}e)(z)\Rangle_{\C^{d}}\dva (z)
\\&=\int_{\Bn}\sum_{j=1}^{d}\Langle f(z),e_{j}\Rangle_{\C^{d}}
\Langle (T_{\mathcal{R}}K_{z}e_{j})(w),e\Rangle_{\C^{d}}\dva (z).
\end{align*}
Finally, this computation gives:
\begin{align*}
(T_{\mathcal{R}}f)(w)&=\sum_{i=1}^{d}
\Langle (T_{\mathcal{R}}f)(w),e_{i}\Rangle_{\C^{d}}e_{i}
\\&=\sum_{i=1}^{d}\sum_{j=1}^{d}\int_{\Bn}\Langle f(z),e_{j}\Rangle_{\C^{d}}
\Langle (T_{\mathcal{R}}K_{z}e_{j})(w),
e_{i}\Rangle_{\C^{d}}\dva (z)e_{i}.
\end{align*}
This shows us that \begin{math}T_{\mathcal{R}}\end{math} is an integral 
operator, with matrix--valued kernel 
\begin{math}M(z,w)\end{math} given by the following relation 
\begin{align*}
\Langle M(z,w)e_{i},e_{j}\Rangle_{\C^{d}} = 
\Langle (T_{\mathcal{R}}K_{z}e_{j})(w),e_{i}
\Rangle_{\C^{d}}.
\end{align*}
That is, 
\begin{align*}
(T_{\mathcal{R}}f)(w)=\int_{\Bn}M(z,w)f(z)\dva (z).
\end{align*}
We now define the 
truncations of this operator. For \begin{math}t\in(0,1)\end{math}, we define the
operator \begin{math}(T_{\mathcal{R}})_{[t]}\end{math} on \begin{math}
A_{\alpha}^{2}\end{math} by:
\begin{align*}
\lp\lp T_{\mathcal{R}}\rp_{[t]}f\rp(w) = \int_{t\Bn}  M(z,w)f(z)\dva (z).
\end{align*}
So that \begin{math}(T_{\mathcal{R}})_{[t]}\end{math} is an integral operator 
with kernel given by: \begin{math}M_{[t]}(z,w)=1_{t\Bn}(z)M(z,w) \end{math}
Let \begin{math} \|\cdot\|_{F}\end{math} denote the Frobenius norm of a 
\begin{math}d\times d\end{math} matrix. We make the following estimation:
\begin{align*}
\|M(z,w)\|_{F}^{2} 
&= \sum_{i=1}^{d}\sum_{k=1}^{d}|
    \Langle M(z,w)e_{k},e_{i}\Rangle_{\C^{d}}|^{2} 
\\&= \sum_{i=1}^{d}\sum_{k=1}^{d}|
    \Langle (T_{\mathcal{R}}K_{z}e_{i})(w),e_{k}\Rangle_{\C^{d}}|^{2}
\\&\leq d \sum_{i=1}^{d}\|(T_{\mathcal{R}}K_{z}e_{i})(w)
    \|_{2}^{2}.
\end{align*}
This gives for any \begin{math}t\in [0,1)\end{math}: 
\begin{align*}
\int_{\Bn}\int_{\Bn} \|M_{[t]}(z,w)\|_{F}^{2}\dva (z)\dva (w) 
&=\int_{\Bn}\int_{t\Bn} \|M(z,w)\|_{F}^{2}\dva (z)\dva (w) 
\\&\leq \int_{\Bn}\int_{t\Bn}d \sum_{k=1}^{d}\|(T_{R}
K_{z}e_{k})(w)\|_{2}^{2}\dva (z)\dva (w) 
\\&\leq d\sum_{k=1}^{d} \int_{t\Bn} \int_{\Bn}\|
(T_{R}K_{z}e_{k})(w)\|_{2}^{2}\dva (w)\dva (z)
\\&=d\sum_{k=1}^{d}\int_{t\Bn}\|T_{R}(K_{z}
e_{k})\|_{A_{\alpha}^{2}}^{2}\dva (z)
\\&\leq d\sum_{k=1}^{d}\int_{t\Bn}\|T_{R}\|^{2}\|K_{z}e_{k}
\|_{A_{\alpha}^{2}}^{2}\dva (z)
\\&\leq d^{2}\|T_{R}\|^{2}\int_{t\Bn}\|K_{z}
\|_{A_{\alpha}^{2}}^{2}\dva (z)
\\&< \infty.
\end{align*}
Thus, by Lemma \ref{lem:vvhs} \begin{math}\lp T_{\mathcal{R}}\rp_{[t]}
\end{math} 
is Hilbert-Schmidt, 
and therefore compact.
So, to show that \begin{math}T_{\mathcal{R}}\end{math} is compact, 
all we need to prove is that:
\begin{align*}
\lim_{t\to 1^{-}}\|T_{\mathcal{R}}-
(T_{\mathcal{R}})_{[t]}\|_{\mathcal{L}(L^{2}_{\alpha})}=0.
\end{align*}
Note that \begin{math}T_{\mathcal{R}}-(T_{\mathcal{R}})_{[t]}\end{math} 
is an integral operator with kernel given by
\begin{math}1_{\Bn\setminus t\Bn}(z)M(z,w)\end{math}. We will use Schur's Test 
and Lemma 
\ref{lem:raiProp3.3} to estimate the norm of this operator. For Schur's test, 
choose \begin{math}\|K_{z}\|_{A_{\alpha}^{2}}^{\epsilon/2}\end{math}
as our test function. If we choose \begin{math}p \text{ such that } 
0 < (p-1)<(n+1)^{-1}\end{math} and 
\begin{math}q\end{math} as conjugate exponent, and \begin{math}\epsilon\in 
(2(p-1)p^{-1},2(n+1)^{-1}p^{-1})\end{math} then we can apply Lemma 
\ref{lem:raiProp3.3}. Also, let \begin{math}G(t,w)=1_{\Bn\setminus t\Bn}(w)
\end{math}. Then 
\begin{align*}
\int_{\Bn} \|G(t,w)M(z,w)\|\|K_{w}
\|_{A_{\alpha}^{2}}^{\epsilon}\dva (w) 
&\simeq \int_{\Bn} |G(t,w)|\sum_{i,j=1}^{d}\left| 
\Langle M(z,w)e_{i},e_{j}\Rangle_{\C^{d}}\right|\|K_{w}
\|_{A_{\alpha}^{2}}^{\epsilon} \dva (w)
\\&\leq d\sum_{i=1}^{d} \int_{\Bn} |G(t,w)| 
\|(T_{\mathcal{R}}K_{z}e_{i})(w)\|_{\C^{d}}\|K_{w}
\|_{A_{\alpha}^{2}}^{\epsilon} \dva (w)
\\&\leq \sum_{i=1}^{d}\|K_{z}
\|_{A_{\alpha}^{2}}^{\epsilon}\lp\sup_{z\in \Bn}
\|T_{-z}e_{i}\|_{A_{\alpha}^{q}}\rp 
\sup_{z\in \Bn}
|\mathcal{J}_{a,b}(z)|^{1/p}
\\&= \|K_{z}\|_{A_{\alpha}^{2}}^{\epsilon}
\sum_{i=1}^{d}\lp\sup_{z\in \Bn}
\|T_{-z}e_{i}\|_{A_{\alpha}^{q}}\rp \sup_{z\in \Bn}
|\mathcal{J}_{a,b}(z)|^{1/p}
\end{align*}
and
{\allowdisplaybreaks
\begin{align*}
\int_{\Bn} \|G(t,z)M(z,w)\|\|K_{z}
    \|_{A_{\alpha}^{2}}^{\epsilon}\dva (z) 
&\simeq \int_{\Bn} |G(t,z)|\sum_{i,j=1}^{d}\left| 
    \Langle M(z,w)e_{i},e_{j}\Rangle_{\C^{d}}\right|\|K_{z}
    \|_{A_{\alpha}^{2}}^{\epsilon} \dva (z)
\\&= \int_{\Bn} |G(t,z)|\sum_{i,j=1}^{d}\left| 
    \Langle e_{j},\overline{(T_{\mathcal{R}}K_{z}e_{i})(w)}\Rangle_{\C^{d}}
    \right| \|K_{z}
    \|_{A_{\alpha}^{2}}^{\epsilon}\dva (z)
\\&\leq d\sum_{i=1}^{d} \int_{\Bn} |G(t,z)| 
    \|(T_{\mathcal{R}}K_{z}e_{i})(w)\|_{\C^{d}}\|K_{z}
    \|_{A_{\alpha}^{2}}^{\epsilon} \dva (z)
\\&\leq \sum_{i=1}^{d}\|K_{w}
    \|_{A_{\alpha}^{2}}^{\epsilon}\lp\sup_{z\in \Bn}
    \|T_{-z}^{*}e_{i}\|_{A_{\alpha}^{q}}\rp 
    \sup_{w\in \Bn}|\mathcal{J}_{a,b}(w)|^{1/p}
\\&= \|K_{w}\|_{A_{\alpha}^{2}}^{\epsilon}
    \sum_{i=1}^{d}\lp\sup_{z\in \Bn}
    \|T_{-z}^{*}e_{i}\|_{A_{\alpha}^{q}}\rp 
    \sup_{w\in \Bn}|\mathcal{J}_{a,b}(w)|^{1/p}.
\end{align*}}
If we choose \begin{math}a,b\end{math} as in Lemma \ref{lem:raiProp3.3}, we have
that \begin{math}\sup_{w\in \Bn}|\mathcal{J}_{a,b}(w)|^{1/p} < \infty\end{math} 
and our hypotheses on \begin{math}\|T_{-z}e\|_{A_{\alpha}^{2}}\end{math}, an 
application of Schur's test gives that \begin{math}
\lim_{t\to 1^{-1}}\|T_{\mathcal{R}}-(T_{\mathcal{R}})_{[t]}
\|_{\mathcal{L}(L_{\alpha}^{2})}=0\end{math}. 
This gives that 
\begin{math}T_{\mathcal{R}}\end{math} is compact, and therefore, 
\begin{math}T\end{math} is compact. 
\end{proof}

\section{Acknowledgements} 
The author would like to thank Michael Lacey for supporting him as a
research assistant for the Spring semester of 2014 (NSF DMS grant 
\#1265570). He would also like to thank Brett Wick for supporting him
as a research assistant for the Summer semester of 2014 (NSF DMS 
grant \#0955432) and for his discussion of the problem. Finally, he 
would like to thank Philip Benge for his helpful comments and 
suggestions. 



\end{document}